\documentclass[12pt]{amsart}
\numberwithin{equation}{section}
\makeatletter

\newcommand{\Rmnum}[1]{\expandafter\@slowromancap\romannumeral #1@}
\makeatother
\def\eps{\varepsilon}

\def\R{\mathbb R}

\def\C{\mathbb C}

\def\RR{\overline{{\mathbb R}^n}}

\def\N{\mathbb N}

\def\card{\operatorname{card}}
\def\capacity{\operatorname{cap}}
\def\deg{\operatorname{deg}}
\def\dim{\operatorname{dim}}

\def\diam{\operatorname{diam}}

\newtheorem{lemma}{Lemma}[section]
\newtheorem{theorem}{Theorem}[section]
\newtheorem{cor}{Corollary}[section]

\theoremstyle{definition}
\newtheorem{definition}{Definition}[section]
\newtheorem{conjecture}{Conjecture}[section]
\theoremstyle{remark}
\newtheorem*{rem}{Remark}

\title[Non-uniformly quasiregular 
maps]{Fatou-Julia theory for non-uniformly quasiregular maps}
\author{Walter Bergweiler}
\email{bergweiler@math.uni-kiel.de}
\address{Mathematisches Seminar,
Christian--Albrechts--Universit\"at zu Kiel,
Lude\-wig--Meyn--Str.~4,
D--24098 Kiel,
Germany}
\thanks{Supported 
by a Chinese Academy of Sciences Visiting Professorship for Senior
International Scientists, Grant No.\ 2010 TIJ10, the Deutsche
Forschungsgemeinschaft, Be 1508/7-1,  the EU Research Training
Network CODY and the ESF Networking Programme HCAA}
\begin{document}
\begin{abstract}
Many results of 
the Fatou-Julia iteration theory of rational functions
extend to uniformly quasi\-regular maps in higher dimensions. 
We obtain results of this type for  certain classes of
quasi\-regular maps which are not uniformly quasi\-regular.
\end{abstract}
\maketitle
\section{Introduction and main results}\label{intro}
Quasiregular maps are a natural generalization of holomorphic 
maps to higher dimensions. It is the purpose of this paper to 
show that certain results of holomorphic dynamics have analogs for
quasi\-regular maps.
We will recall the definition
and basic properties
of quasi\-regular maps in section~\ref{quasi},
defining in particular terms like the dilatation $K(f)$ and the inner 
dilatation $K_I(f)$ of a quasi\-regular map $f$ that are used in 
the following.

An important result about quasi\-regular maps is
Rickman's~\cite{Rickman80,Rickman85}
analog of Picard's theorem. He showed that there
exists a constant $q=q(n,K)$ such that if 
$a_1,\dots,a_q\in\R^n$ are distinct and $f:\R^n\to \R^n\setminus
\{a_1,\dots,a_q\}$ is $K$-quasi\-regular, then $f$ is constant.
Note that Picard's  theorem says that $q(2,1)=2$.

Miniowitz~\cite{Miniowitz82} used an extension of the
Zalcman lemma~\cite{Zalcman75} to quasi\-regular maps to obtain
an analog of Montel's theorem from Rickman's result.
Given the central role of Montel's theorem in holomorphic
dynamics, it is seems clear that Miniowitz's theorem will
be important in quasi\-regular dynamics. However, 
in order to apply this result
to the family $\{f^j\}$ of iterates a quasi\-regular map~$f$, one
has to assume that all $f^j$ are $K$-quasi\-regular
with the same~$K$. Quasiregular maps with this property are
called \emph{uniformly quasi\-regular}.
For uniformly quasi\-regular self-maps of
the one point compactification
$\overline{\R^n}:=\R^n\cup \{\infty\}$ of~$\R^n$
an iteration
theory in the spirit of Fatou and Julia has been developed
by Hinkkanen, Martin, Mayer and others~\cite{Hinkkanen04,Martin97,Mayer97};
see~\cite[Section~4]{Bergweiler10},
\cite[Chapter~21]{Iwaniec01} and \cite[Chapter~4]{Siebert04}
for surveys.

As in the classical case of rational maps, the 
Julia set $J(f)$ of a uniformly quasi\-regular map 
$f:\RR\to\RR$ is defined as the set of all points where the
family of iterates fails to be normal.
Assuming that the degree of $f$ is at least $2$
one finds that $J(f)$ is perfect; in particular, $J(f)\neq\emptyset$. 
Here the degree $\deg(f)$ of a (not necessarily uniformly) quasi\-regular map
$f:\RR\to\RR$ is defined as the maximal cardinality
of the preimage of a point;
that is, 
$$\deg(f):=\max_{x\in\RR}\card f^{-1}(x),$$ 
where $\card A$ denotes the cardinality of a set~$A$.

For $x\in \RR$ 
we define the forward orbit 
$O^+(x):=\{f^j(x):j\in\N\}$ and for 
$X\subset \RR$ we put 
$O^+(X):=\bigcup_{x\in X}O^+(x)$.
One direct consequence of Miniowitz's theorem is the so-called
expansion property which says that if $U$ is an
open set intersecting the Julia set, then $\RR\setminus O^+(U)$
is finite. In fact, this set contains at most $q(n,K)$ points,
provided $K(f^j)\leq K$ for all $j\in\N$.

We refer to
the papers mentioned above -- and the references cited 
therein -- for further results about the dynamics of 
uniformly quasi\-regular maps.

Sun and Yang~\cite{SunYang99,SunYang00,SunYang01}
showed that in dimension $2$ some results of the Fatou-Julia theory
still hold even for non-uniformly quasi\-regular
maps, provided the degree exceeds the dilatation. However, 
the definition of the Julia set via
non-normality is not adequate here. Instead Sun and Yang used 
the expansion property to define the Julia set. They thus defined
the Julia set $J(f)$ of a quasi\-regular 
self-map $f$ of the Riemann sphere $\overline{\C}$ 
as the set of all $z\in\overline{\C}$ such that $\overline{\C}
\setminus O^+(U)$ contains at most two points, for every
neighborhood $U$ of~$z$.
They showed that if $\deg(f)>K(f)$, then $J(f)\neq\emptyset$,
and many results of the Fatou-Julia theory hold.
For an exposition of their results we refer 
to~\cite[Section 5]{Bergweiler10}.

There have been only a few papers concerned with the 
the dynamics of non-uniformly quasi\-regular maps in higher
dimensions. In~\cite{Bergw,BergwErem,Bergweiler09} 
certain quasi\-regular maps $f:\R^n\to\R^n$ with an essential
singularity at infinity were considered. Such maps can be 
thought of as analogs of transcendental entire functions.
In contrast, 
a quasi\-regular map $f:\R^n\to\R^n$ is said to be of 
\emph{polynomial type} if $\lim_{x\to\infty}f(x)=\infty$. Such a
map  $f$ extends to a quasi\-regular self-map
of $\RR$ by putting $f(\infty)=\infty$.
The dynamics of such maps where studied by Fletcher and 
Nicks~\cite{FletcherNicks} who proved that if $\deg(f)>K_I(f)$,
then 
the boundary of the 
\emph{escaping set} $I(f):=\{x\in\R^n: f^j(x)\to\infty\}$
has many properties usually associated with the Julia set.
Note that $J(f)=\partial I(f)$ for non-linear
polynomials $f:\C\to\C$, 
as well as transcendental entire functions~\cite{Eremenko89}.

We shall be concerned with quasi\-regular self-maps of
$\RR$ which need not be of polynomial type.
Such maps can be considered as analogs
of rational functions.
In order to state our first result, we need to introduce sets of 
capacity zero; cf. \cite[section II.10]{Rickman93}.
For an open set $G\subset\R^n$ and a 
non-empty compact subset $C$ of $G$ the pair
$(G,C)$ is called a \emph{condenser} and its  \emph{capacity}
$\capacity (G,C)$ is defined by
\[
\capacity (G,C):=\inf_u\int_G\left|\nabla u\right|^n dm,
\]
where the infimum is taken over all non-negative functions
$u\in C^\infty_0(G)$ satisfying $u(x)\geq 1$ for all $x\in C$.
(Here $C^\infty_0(G)$ may be replaced by the
Sobolev space $W^1_{n,\text{loc}}(G)$, which also appears in the 
definition of quasi\-regularity;  cf.\ section~\ref{quasi}.)

It turns out \cite[Lemma III.2.2]{Rickman93}
 that if $\capacity (G,C)=0$ 
for some bounded open set $G$ containing~$C$, 
then  $\capacity (G',C)=0$ 
for every bounded open set $G'$ containing~$C$.
In this case we say that $C$ is of  \emph{capacity zero} and denote
this by $\capacity C=0$.
Otherwise we say that $C$ has \emph{positive capacity} and 
write $\capacity C>0$.
Note that this does not mean that 
$\capacity C$ is a positive number.
 (The capacity is defined for condensers, not for sets.) 
However, 
we mention that
Vuorinen~\cite{Vuorinen83} has introduced a set function $c$ satisfying
$c(C)>0$ if and only if $\capacity C>0$.
M\"obius transformations preserve the capacity of a condenser
and hence preserve sets of capacity zero, leading to an obvious
 extension of the definition
to subsets of $\RR$; see~\cite[Section~1.3]{Reshetnyak} for  the 
definition and a discussion of
M\"obius transformations.

We mention that sets of capacity zero are totally disconnected
\cite[Corollary III.2.5]{Rickman93}
and in fact of Hausdorff dimension zero \cite[Corollary VII.1.15]{Rickman93};
see also Lemma~\ref{cap-haus} below for a stronger statement 
involving Hausdorff measure.
On the other hand, a finite set has capacity zero.

\begin{theorem}\label{thm1}
Let $f:\RR\to\RR$ be quasi\-regular. Suppose that $\deg(f)>K_I(f)$.
Then there exists $x\in\RR$ such that 
\begin{equation}\label{1c}
\capacity\!\left(\RR\setminus O^+(U)\right)=0
\end{equation}
for every neighborhood $U$ of~$x$.
\end{theorem}
As in~\cite{FletcherNicks,SunYang99} the winding map (cf.
\cite[Section I.3.1]{Rickman93}) shows that the 
hypothesis  that
$\deg(f)>K_I(f)$ cannot be weakened to  $\deg(f)\geq K_I(f)$.

We note that 
if $f$ is uniformly quasi\-regular
and $\deg(f)\geq 2$, then the
hypothesis of Theorem~\ref{thm1} is satisfied for some
iterate of~$f$.
The hypothesis of Theorem~\ref{thm1} and subsequent
theorems could be weakened
to $\deg(f^p)> K_I(f^p)$ for some $p\in\N$ in order to cover all
uniformly quasi\-regular maps, but for simplicity
we restrict ourselves to the case $p=1$.

We mention that the composition of a uniformly quasi\-regular map
with a M\"obius transformation need not be uniformly quasi\-regular.
In contrast, the hypothesis of Theorem~\ref{thm1} is preserved
under compositions with M\"obius transformations. This yields many
examples of quasi\-regular maps satisfying the hypothesis of 
Theorem~\ref{thm1} which are not uniformly quasi\-regular.

Following Sun and Yang we define the Julia set as follows.
\begin{definition}\label{defJ}
Let 
$f:\RR\to\RR$ be quasi\-regular.  Then
the set of all $x\in\RR$ such that~\eqref{1c} 
holds for every neighborhood $U$ of $x$ is called the 
\emph{Julia set} of $f$ and denoted by $J(f)$.
\end{definition}
Theorem~\ref{thm1} says that $J(f)\neq\emptyset$ if $\deg(f)>K_I(f)$.
As in the case of rational functions 
it is easy to see that $J(f)$ is closed and completely invariant;
cf.~\cite[Theorem~3.2.4]{Beardon91},
\cite[Lemma~4.3]{Milnor06}
 or~\cite[Section 25]{Steinmetz93}.
Here a set $A$ is called \emph{completely invariant} (under $f$) if
$x\in A$ implies that $f(x)\in A$, and vice versa.
It follows that $J(f)$ has empty interior unless 
$J(f)=\RR$; cf.~\cite[Theorem~4.2.3]{Beardon91},
\cite[Corollary~4.11]{Milnor06}
or~\cite[Section 30]{Steinmetz93}.

Definition~\ref{defJ} is justified by the following result.
\begin{theorem}\label{thm0}
For a uniformly quasi\-regular map $f:\RR\to\RR$
the definition of $J(f)$ using non-normality
coincides with the one  given in Definition~\ref{defJ}.
\end{theorem}

A point  $\xi\in\RR$ is called \emph{periodic} if there exists $p\in\N$ such
that $f^p(\xi)=\xi$. The smallest $p$ with this property is called
the \emph{period} of~$\xi$. 
We denote by $\chi$ the chordal metric on~$\RR$,
obtained via stereographic projection from the unit sphere
in~$\R^{n+1}$. 
Following Sun and Yang~\cite[Definition~4]{SunYang01} we say that
a periodic point $\xi$ of period $p$
is \emph{attracting} if there
exists $c\in (0,1)$ and a neighborhood $U$ of $\xi$ such that
$\chi(f^p(z),\xi)<c\, \chi(z,\xi)$ for all $z\in U$.
Similarly we say that $\xi$ is \emph{repelling} if
$\chi(f^p(z),\xi)>c \,\chi(z,\xi)$ for some $c>1$ and
all $z$ in some neighborhood of~$\xi$.

We note that other definitions of 
attracting and repelling have been used for
uniformly quasi\-regular maps (cf.~\cite{Hinkkanen04}
for a discussion), but
all definitions have in common that an attracting periodic point 
of period $p$ has a neighborhood where the iterates of $f^p$ 
converge uniformly to it.
For an attracting periodic point $\xi$ of period $p$
the set 
$$A(\xi):=\{x\in\RR: \lim_{j\to\infty} f^{pj}(x)= \xi\},$$
called the \emph{attracting basin} of~$\xi$, thus contains
a neighborhood of~$\xi$.
\begin{theorem}\label{thm3}
Let $f:\RR\to\RR$ be quasi\-regular with $\deg(f)>K_I(f)$.
If $\xi$ is an attracting periodic point of~$f$,
then $J(f)\cap  A(\xi)=\emptyset$ and
$J(f)\subset\partial A(\xi)$.
\end{theorem}
For rational functions 
and, more generally, uniformly quasi\-regular
maps we have $J(f)=\partial A(\xi)$;
see~\cite[Corollary 4.12]{Milnor06}.
As shown in
\cite[Example~5.3]{Bergweiler10}, this need not be the case in the present
setting.

For a map $f:\RR\to\RR$ the \emph{exceptional set}
$E(f)$ is defined as the set of all $x\in\RR$ for which 
the backward orbit $O^-(x):=\bigcup_{j=1}^\infty f^{-j}(x)$ is finite.
\begin{theorem}\label{thm4}
Let $f:\RR\to\RR$ be quasi\-regular with $\deg(f)>K_I(f)$.
Then $E(f)$ is finite and consists of attracting periodic 
points.
In particular, 
$E(f)$ does not intersect $J(f)$.
\end{theorem}
This result is standard for rational functions;
see \cite[Section~4.1]{Beardon91},
\cite[Lemma 4.9]{Milnor06}
 or \cite[Section~31]{Steinmetz93}.
For uniformly quasi\-regular maps it
can be found in, e.g., \cite[pp.~64--65]{Siebert04}.

Quasiregular maps are H\"older continuous.
For the  analogs of some further key results of complex dynamics
we require the stronger hypothesis of Lipschitz continuity.
This condition is satisfied
for many examples of uniformly quasi\-regular maps. 
We  also note that uniformly quasi\-regular maps are 
Lipschitz continuous at fixed points~\cite[Lemma~4.1]{Hinkkanen04}.
\begin{theorem}\label{thm5}
Let $f:\RR\to\RR$ be quasi\-regular with $\deg(f)>K_I(f)$.
Suppose that $f$ is Lipschitz continuous.
If $U$ is  an open set intersecting $J(f)$,
then
$O^+(U)\supset \RR\setminus E(f)$
and $O^+(U\cap J(f))=J(f)$.
\end{theorem}
\begin{theorem}\label{thm7}
Let $f$ be as in Theorem~\ref{thm5}.
Then $J(f)=\overline{O^-(x)}$ for every $x\in J(f)$
and
$J(f)\subset\overline{O^-(x)}$ for every $x\in \RR\setminus E(f)$.
\end{theorem}
Theorems~\ref{thm5} and~\ref{thm7} 
are well-known for rational functions;
see~\cite[Theorems~4.2.5 and 4.2.7]{Beardon91},
\cite[Theorem~4.10 and Corollary~4.13]{Milnor06}
or~\cite[Sections 28 and 32]{Steinmetz93}.
For uniformly quasi\-regular maps these results are -- as already 
mentioned -- consequences of Miniowitz's theorem and
can be found in, e.g., \cite[Section~3]{Hinkkanen04}.

We denote the Hausdorff dimension of a subset  $A$ of
$\R^n$ by $\dim A$.
\begin{theorem}\label{thm6}
Let $f$ be as in Theorem~\ref{thm5}.
Then $\dim J(f)>0$.
\end{theorem}
For rational functions 
Theorem~\ref{thm6} is 
due to Garber; see~\cite{Garber}, 
\cite[Section~10.3]{Beardon91} or \cite[Section~168]{Steinmetz93}.
For uniformly quasi\-regular maps it was recently proved
by Fletcher and Nicks~\cite{FletcherNicks11}.

We conjecture that the hypothesis that $f$ is Lipschitz continuous
can be omitted in Theorems~\ref{thm5} and~\ref{thm7},
but not in Theorem~\ref{thm6}.
However, we conjecture that without this hypothesis
we still have $\capacity J(f)>0$.
We prove that this is the case under an additional assumption
involving the \emph{branch set} $B_f$ which is defined as 
the set of all points where $f$ is not locally injective;
cf.~section~\ref{quasi}.
\begin{theorem}\label{thm8}
Let $f:\RR\to\RR$ be quasi\-regular with $\deg(f)>K_I(f)$.
Suppose that $J(f)\cap B_f=\emptyset$. Then $\capacity J(f)>0$.
\end{theorem}

This paper is organized as follows.
In section~\ref{quasi} we 
recall the definition of quasi\-regular 
maps and in section~\ref{averages} we state some results about
averages of counting functions which play a key role
in the proof of Theorem~\ref{thm1}.
 These results will be proved in section~\ref{mr},
using some lemmas about the capacities of condensers
and the moduli of path families given
before
in section~\ref{moduli}. Theorems~\ref{thm1}--\ref{thm3}
are then proved in section~\ref{proofthm1} and
Theorem~\ref{thm4} is proved in section~\ref{proofthm4}.
In section~\ref{hausdorffinv} we obtain some results
about the Hausdorff measure of invariant sets and use them
in section~\ref{proofthm5} to prove Theorems~\ref{thm5}--\ref{thm8}.
In section~\ref{locdist} we prove a result about the local
distortion of quasi\-regular maps which generalizes a result
used in the proof of Theorem~\ref{thm4} and which may be of
independent interest.
In section~\ref{conj} we give some evidence for the conjecture
made above that 
Theorems~\ref{thm5} and~\ref{thm7} hold without
the hypothesis of Lipschitz continuity.
We also show that the conclusion of these theorems holds 
under some different hypothesis.

\section{Quasiregular maps}\label{quasi}
We denote the (Euclidean) norm of  a point $x\in \R^n$ by~$|x|$.
For $a\in\R^n$ and $r>0$ let $B(a,r):=\{x\in\R^n:|x-a|<r\}$
be the open ball,
$\overline{B}(a,r)$ the closed ball and $S(a,r)=\partial
B(a,r)$ the sphere
 of radius $r$ centered at~$a$.
We write $B(r)$, $\overline{B}(r)$ and $S(r)$
 instead of $B(0,r)$, $\overline{B}(0,r)$
and $S(0,r)$.
Sometimes we will emphasize the dimension by writing
$B^n(a,r)$, $S^{n-1}(a,r)=\partial B^n(a,r)$, etc.
With the stereographic projection $\pi:S^n(1)\to\RR$ the
chordal metric $\chi$
already mentioned
is given by $\chi(x,y)=|\pi^{-1}(x)-\pi^{-1}(y)|$.
(Instead of the chordal metric, 
one could also use the spherical metric.)
Balls with respect to the 
chordal metric are denoted by a subscript~$\chi$; that is,
$B_\chi(a,r):=\{x\in\RR:\chi(x,a)<r\}$.

We recall the definition of quasi\-regularity; see
Rickman's monograph~\cite{Rickman93} for 
more details.
Let $n\geq 2$ and let $\Omega\subset \R^n$ be a domain.
For $1\leq p<\infty$ the Sobolev space  $W^1_{p,\text{loc}}(\Omega)$ is defined
as the set of 
functions  $f=(f_1,\dots,f_n):\Omega\to \R^n$ for which
all first order weak partial derivatives
$\partial_k f_j$ exist and are locally in $L^p$.
It turns out that a continuous map
$f$ is in $W^1_{p,\text{loc}}(\Omega)$ if and only if
all $f_j$ are absolutely continuous on almost all lines
parallel to the coordinate axes, with all partial
derivatives locally $L^p$-integrable.
For us only the case $p=n$ will be of interest.

A~continuous map
$f\in W^1_{n,\text{loc}}(\Omega)$
is called \emph{quasi\-regular}
if there exists
a constant $K_O\geq 1$ such that
\begin{equation}\label{1a}
|Df(x)|^n\leq K_O J_f(x)
\quad \mbox{ a.e.},
\end{equation}
where $Df(x)$ denotes the derivative,
\[
|Df(x)|:=\sup_{|h|=1} |Df(x)(h)|
\]
its norm, and $J_f(x)$ the Jacobian determinant.
With
\[
\ell(Df(x)):=\inf_{|h|=1}|Df(x)(h)|
\]
the condition that~\eqref{1a} holds
for some $K_O\geq 1$ is equivalent to the condition that
\begin{equation}\label{1b}
J_f(x)\leq K_I\ell(Df(x))
\quad \mbox{ a.e.},
\end{equation}
for some $K_I\geq 1$. The smallest constants
$K_O$ and $K_I$ for which~\eqref{1a} and~\eqref{1b} hold are
called the \emph{outer and inner dilatation} of $f$ and
and denoted by $K_O(f)$ and $K_I(f)$.
Moreover,
$K(f):=\max\{K_I(f),K_O(f)\}$ is called the (maximal) \emph{dilatation} of~$f$.
We say that $f$ is $K$-\emph{quasi\-regular} if $K(f)\leq K$.

If $f$ and $g$ are quasi\-regular, with $f$ defined in the 
range of~$g$, then
$f\circ g$ is also quasi\-regular and~\cite[Theorem~II.6.8]{Rickman93}
\begin{equation}\label{KIfg}
K_I(f\circ g)\leq K_I(f)K_I(g)
\quad\text{and}\quad
K_O(f\circ g)\leq K_O(f)K_O(g)
\end{equation}
so that $K(f\circ g)\leq K(f)K(g)$.

As already mentioned, many properties  of holomorphic  functions
carry over to quasi\-regular maps. Here we only note that 
non-constant quasi\-regular maps are open and discrete.
We refer to the monographs~\cite{Reshetnyak,Rickman93} for
a detailed treatment of quasi\-regular maps.

The \emph{local index} $i(x,f)$ 
of a quasi\-regular map $f:\Omega\to\R^n$ at a point $x\in\Omega$ 
is defined by
$$i(x,f) :=\inf_U \sup_{y\in\R^n}\card f^{-1}(y),$$ 
where the infimum is taken over all neighborhoods $U\subset 
\Omega$ of~$x$.
We thus have $i(x,f)=1$ if and only if $f$ is injective in
a neighborhood of~$x$.
The \emph{branch set} $B_f$ already mentioned in the introduction
consists of all $x\in \Omega$ for which $i(x,f)\geq 2$.

Quasiregularity can be defined more generally for maps between
Riemannian manifolds. Here we consider only the case that
the domain or range are equal to (or contained in)
$\RR$.
It turns out that for a domain $\Omega\subset\overline{\R^n}$
a non-constant continuous map $f:\Omega \to \overline{\R^n}$
is quasi\-regular if $f^{-1}(\infty)$ is discrete and
if $f$ is quasi\-regular
in $\Omega\setminus (f^{-1}(\infty)\cup\{\infty\} )$.

\section{Averages of counting functions}\label{averages}
For a quasi\-regular map $f:\Omega\to \RR$, a compact
subset $E$ of $\Omega$  and $y\in\RR$ we 
denote by $n(E,y)$ the number of $y$-points of $f$ in~$E$,
counted according to multiplicity.
Thus
$$
n(E,y)=\sum_{x\in f^{-1}(y) \cap E} i(x,f).
$$
We will 
consider
the average value of $n(E,y)$ over a sphere $S(z,t)$ 
and denote this by $\nu(E,S(z,t))$. 
Denoting the normalized
$d$-dimensional Hausdorff measure by $H^d$
and putting 
$\omega_{d}=H^{d}(S^{d}(1))$ 
for $d\in \N$ we 
thus have
$$
\nu(E,S(z,t))=\frac{1}{\omega_{n-1}t^{n-1}}\int_{S(z,t)} n(E,y)dH^{n-1}(y).
$$
We will mainly be concerned with the case that $E=\overline{B}(r)$.
In this case we use the notation $n(r,y)$ and $\nu(r,S(z,t))$ 
instead of $n\!\left(\overline{B}(r),y\right)$ and 
$\nu\!\left(\overline{B}(r),S(z,t)\right)$.

The following result 
is obtained by careful inspection and suitable modification
of a result of Mattila and Rickman~\cite[Lemma~3.3]{Mattila79}.
We shall give the proof in section~\ref{mr}.
\begin{theorem}\label{thm-mr}
There exists a constant $C$ depending only on the dimension $n$ 
such that if 
$F\subset B^n(z,t/2)$ is a compact set of positive capacity,
$\theta>1$ and 
$f:B^n(\theta r)\to \RR\setminus F$
is quasi\-regular, then
\begin{equation}\label{nuSzt}
\nu(r,S(z,t))\leq 
C\frac{K_I(f)}{(\log\theta)^{n-1}
\capacity\!\left(B^n(t),F\right)}.
\end{equation}
\end{theorem}
The average of $n(E,y)$ over $\RR$ is denoted by $A(E)$.
Identifying $\RR$ with $S^{n}(1)$ we thus have
$$
A(E)=\frac{1}{\omega_{n}}\int_{S^{n}(1)} n(E,y)dH^{n}(y).
$$
Similarly as before
we write $A(r)$ instead of $A\!\left(\overline{B}(r)\right)$, and sometimes we
include the map $f$ by writing $A(r,f)$.

It is shown in \cite[Lemma IV.1.7]{Rickman93}
that $\nu(r,S(z,t))$ and $A(r)$ are comparable in the following
sense.
\begin{lemma}\label{la-nu-A}
There exists a constant $Q$ depending only on the dimension $n$ 
such that if
$Y$ is an $(n-1)$-sphere of spherical radius $u\leq \pi/4$,
if 
$R>\theta r>r>0$ and
if $f:B^n(R)\to \RR$ is quasi\-regular, 
then
$$
\nu(r/\theta,Y)-
Q\frac{K_I(f)| \log u|^{n-1}}{(\log\theta)^{n-1}}
\leq A(r)
\leq \nu(\theta r,Y)+
Q\frac{K_I(f)  |\log u|^{n-1}}{(\log\theta)^{n-1}}.
$$
\end{lemma}
Noting that given a set $F$ of positive capacity and $t>0$
we can find  a subset of $F$ which has 
positive capacity and is 
contained in a ball of radius $t/2$, we obtain the following 
result from Theorem~\ref{thm-mr} and Lemma~\ref{la-nu-A}.
\begin{theorem}\label{thm-mr2}
Let $F\subset \RR$ be a set of positive capacity and let
$\theta>1$. Then there exists a constant $C$ depending only
on~$n$, $F$ and $\theta$ such that if $f:B^n(\theta r)\to \RR\setminus F$
is quasi\-regular, then
$A(r,f)\leq C\, K_I(f)$.
\end{theorem}
Clearly, it is irrelevant here  that the balls considered are
centered at $0$ so that if $a\in\R^n$ and
$f:B(a,\theta r)\to\R^n\setminus F$ is quasi\-regular, 
then $A\!\left(\overline{B}(a,r),f\right)\leq C\, K_I(f)$.
Similarly, 
we may consider balls with respect to the chordal metric
and obtain
$A\!\left(\overline{B}_\chi(a,r),f\right)\leq C\, K_I(f)$
if 
$a\in\RR$, $0<r<\theta r<2$
and 
$f:B_\chi(a,\theta r)\to\RR\setminus F$ is quasi\-regular.

\section{Capacity and the modulus of a path family}\label{moduli}
The modulus of a path family is a major tool in the study
of quasi\-regular maps. We review this concept only briefly;
see~\cite[Chapter II]{Rickman93} and~\cite[Chapter~2]{Vuorinen88}
for more details.
Let $\Gamma$ be a family of paths in $\R^n$.
We say that a non-negative Borel function $\rho:\R^n\to\R\cup\{\infty\}$
is \emph{admissible} if $\int_\gamma \rho\; ds\geq 1$ for all
locally rectifiable 
paths $\gamma\in \Gamma$ and denote by $\mathcal{F}(\Gamma)$ the family 
of all admissible Borel functions. Then
$$
M(\Gamma):=\inf_{\rho\in \mathcal{F}(\Gamma)}\int_{\R^n}\rho^n\; dm
$$
is called the \emph{modulus} of $\Gamma$. For the extension
to families of paths in $\RR$ we refer to~\cite[pp.~53--54]{Vuorinen88}.

For a domain $G\subset \RR$ and sets $E,F\subset \overline{G}$ 
we denote by $\Delta(E,F;G)$ the family of all paths which have
one endpoint in~$E$, one endpoint in $F$ and which are in $G$
otherwise. The connection with capacity is given by the following
result~\cite[Proposition II.10.2]{Rickman93}.
\begin{lemma}\label{mcf1}
Let $G\subset \R^n$ be open and $C\subset G$ compact.
Then 
$$\capacity (G,C)=M(\Delta(C,\partial G;G)).$$
\end{lemma}
As an example we mention that for $0<r<s$ we have~\cite[p.~28]{Rickman93}
\begin{equation}\label{capring}
\capacity \!\left(B(s),\overline{B}(r)\right)
=M\!\left(\Delta\!\left(S(r),S(s);B(s)\!\setminus
\!\overline{B}(r)\right)\right)=\omega_{n-1}\left(\log\frac{s}{r}\right)^{1-n} .
\end{equation}
For two path families $\Gamma_1$ and $\Gamma_2$ we write
$\Gamma_1<\Gamma_2$ if every $\gamma\in \Gamma_2$ has a 
subpath belonging to $\Gamma_1$. As Ahlfors~\cite[p.~54]{Ahlfors}
puts it: $\Gamma_2$ has fewer and longer arcs.
The following lemma~\cite[p.~26]{Rickman93}
 follows directly from the definition.
\begin{lemma}\label{mcf2}
If $\Gamma_1<\Gamma_2$, then $M(\Gamma_1)\geq M(\Gamma_2)$.
\end{lemma}
We note that it follows from the definition of capacity,
or from Lemma~\ref{mcf2} and~\eqref{capring}, that
\begin{equation}\label{capsub}
\capacity(C,G)\geq \capacity(C,G')
\quad\text{if}\ G\subset G'.
\end{equation}

The next lemma is known as V\"ais\"al\"a's 
inequality~\cite[Theorem~II.9.1]{Rickman93}.
\begin{lemma}\label{mcf3}
Let $f$ be quasi\-regular in a domain $\Omega\subset \RR$,
let $\Gamma^*$ be a path family in $\Omega$ and let 
$\Gamma$ be a path family in~$\RR$.
Suppose that there exists $m\in\N$ such that for every
path $\beta:I\to \RR$ in $\Gamma$ there are paths
$\alpha_1,\dots,\alpha_m$ in $\Gamma^*$ such that 
$f\circ\alpha_j\subset \beta$ for all $j$ and such that
for every $x\in \Omega$ 
and $t\in I$ 
the equality
$\alpha_j(t)=x$ holds for at most $i(x,f)$ indices~$j$.
Then
$$
M(\Gamma)\leq \frac{K_I(f)}{m}M(\Gamma^*).
$$
\end{lemma}
The following result~\cite[Theorem~II.10.11]{Rickman93} is
a consequence of Lemma~\ref{mcf3}.
\begin{lemma}\label{mcf3a}
Let $f:\Omega\to\R^n$ be quasi\-regular, let
$(G,C)$ be a condenser in $\Omega$ and put
$m:=\inf_{y\in f(C)}n(C,y)$. 
Then
$$
\capacity (f(G),f(C))\leq\frac{K_I(f)}{m} \capacity (G,C).
$$
\end{lemma}

As mentioned, the proof of Theorem~\ref{thm-mr} follows the 
arguments of Mattila and Rickman~\cite{Mattila79}.
The following lemma is taken from their paper~\cite[Lemma~3.2]{Mattila79}.
\begin{lemma}\label{mcf4}
Let $n\geq 2$ and $0<u<v<\infty$. For
$F_1\subset \overline{B^n}(u)$ and  
$F_2\subset S^{n-1}(v)$.
define the path families
$\Sigma_{12}:=\Delta(F_1,F_2;B(v))$,
$\Sigma_{1}:=\Delta(F_1,S(v);B(v))$ and
$\Sigma_{2}:=\Delta(F_2,S(u);B(v)\!\setminus\! \overline{B}(u))$.
Then
$$
M(\Sigma_{12})\geq 3^{-n} \min\{M(\Sigma_{1}),M(\Sigma_{2}),c_n\log(v/u)\},
$$
where $c_n$ depends only on~$n$.
\end{lemma}
Note that with the terminology of Lemma~\ref{mcf4}
we have $M(\Sigma_{1})=\capacity(B(v),F_1)$ by Lemma~\ref{mcf1}.

The next lemma is implicit in the proof of~\cite[Lemma~3.3]{Mattila79},
but for completeness we include the proof.
\begin{lemma}\label{mcf5}
For $n\in\N$ there exist positive constants $\alpha$ and $\beta$
such that if $r>0$ and $A\subset S^{n-1}(r)$ is compact, then
$$
M\!\left(\Delta\!\left(S(r/2),A;B(r)\!\setminus\! \overline{B}(r/2)
\right) \right)
\geq \alpha \left(\log \left(\frac{\beta r^{n-1}}{H^{n-1}(A)}
\right)\right)^{1-n}.
$$
\end{lemma}
Here the right hand side is understood to be $0$ if 
$H^{n-1}(A)=0$.
\begin{proof}[Proof of Lemma~\ref{mcf5}]
By a result of Gehring~\cite[Lemma~1]{Gehring76}
we have 
$$
M\!\left(\Delta\!\left(S(r/2),A;B(r)\!\setminus \!\overline{B}(r/2)
\right)\right)
=
\frac12 M\!\left(\Delta\!\left(S(r/2)\cup S(2r),A;B(2r)\!\setminus\!
 \overline{B}(r /2)
\right)\right).
$$
Thus
$$
\begin{aligned}
M\!\left(\Delta\!\left(S(r/2),A;B(r)\!\setminus\! \overline{B}(r/2)\right)\right)
&=\frac12\capacity\!\left(B(2r)\!\setminus\! \overline{B}(r/2),A\right)
\\ &
\geq \frac12\capacity (B(2r),A)
\end{aligned}
$$
by Lemma~\ref{mcf1} and~\eqref{capsub}.

We may assume that $H^{n-1}(A)>0$ and
denote by $A^*$ the spherical symmetrization of~$A$;
that is, 
using the notation $e_k$ for the \mbox{$k$-th} unit vector 
we put 
$A^*=S(r)\cap \overline{B}(r e_n,s)$, where 
$s$ is chosen such
that $H^{n-1}(A)=H^{n-1}(A^*)$.
By a result of Sarvas~\cite{Sarvas72} we have 
$$
\capacity (B(2r),A)\geq \capacity (B(2r),A^*).
$$
Combining the last two estimates we obtain
\begin{equation}\label{m1}
M\!\left(\Delta\!\left(S(r/2),A;B(r)\!\setminus \!\overline{B}(r/2)
\right)\right)
\geq
\frac12 \capacity (B(2r),A^*).
\end{equation}
We note that the modulus is invariant 
under  translations. With $T(x)=x-r e_n$ we thus have
\begin{equation}\label{m2}
\capacity (B(2r),A^*)
=\capacity (T(B(2r)),T(A^*))
\geq  \capacity (B(3r),T(A^*))
\end{equation}
by~\eqref{capsub}.
Now there exists $c>0$ such that
$$
(\diam A^*)^{n-1}\geq cH^{n-1}(A^*)=cH^{n-1}(A),
$$
where $\diam A^*$ denotes the diameter of $A^*$.
Thus there exists $a\in A^*$ with
$$
|T(a)|=|a-re_n|\geq \frac{1}{2}\left(cH^{n-1}(A)\right)^{1/(n-1)}.
$$
Since $T(A^*)$ is connected and $0\in T(A^*)$, the extremality
of the Gr\"otzsch condenser 
$$
E_{G}(t):=(B^n(1),[0,t e_1])
$$
yields~\cite[Lemma~III.1.9]{Rickman93} 
$$
\capacity (B(3r),T(A^*))
\geq  \capacity E_G\!\left(\frac{|T(a)|}{3r}\right).
$$
Combining this with the estimate~\cite[Lemma~III.1.2]{Rickman93}
$$
\capacity E_{G}(t)\geq \omega_{n-1} \left(\log 
\frac{\lambda_n}{t}\right)^{1-n},
$$ 
where $\lambda_n$ depends only on~$n$, we obtain
\begin{equation} \label{m3}
\begin{aligned}
\capacity (B(3r),T(A^*))
\geq&\ \omega_{n-1}\left(\log \frac{3\lambda_nr}{|T(a)|}\right)^{1-n}\\
\geq&\ 
\omega_{n-1}\left( \log \frac{6\lambda_n r}{\left(cH^{n-1}(A)\right)^{1/(n-1)}} \right)^{1-n}\\
=&\
2\alpha
\left( \log \left(\frac{\beta r^{n-1}}{H^{n-1}(A)}\right)\right)^{1-n}
\end{aligned}
\end{equation}
for suitable constants $\alpha$ and $\beta$ depending only on~$n$.
The conclusion follows from~\eqref{m1}, \eqref{m2}
and~\eqref{m3}.
\end{proof}
We conclude this section with the following lemma
already mentioned in the 
introduction; see~\cite[Corollary VII.1.15]{Rickman93}.
\begin{lemma}\label{cap-dim}
Let $X\subset\R^n$ be compact. If $\dim X>0$, then
 $\capacity X>0$.
\end{lemma}
A strengthened form of Lemma~\ref{cap-dim} is given by
Lemma~\ref{cap-haus} below.

\section{Proof of Theorem~\ref{thm-mr}}\label{mr}
Without loss of generality we may assume that $z=0$.
For $k\in\N$ let 
$$
A_k:=\{y\in S(t):n(r,y)=k\}
\quad\text{and}\quad
B_k:=\{y\in S(t):n(r,y)\geq k\}.
$$
Let $B_k'\subset B_k$ be compact with
$H^{n-1}(B_k')\geq H^{n-1}(B_k)/2$ and
consider the path family  $\Gamma_k:=\Delta(F,B_k';B(t))$.
Each $\gamma\in \Gamma_k$ has $k$ liftings
$\alpha_1,\dots,\alpha_k$ under $f$ which
connect a point in $\overline{B}(r)$ to $S(\theta r)$
and have the properties stated in Lemma~\ref{mcf3};
cf.~\cite[Section II.3]{Rickman93}. Let $\Gamma_k^*$
be the family of all these liftings. Then
$$
M(\Gamma_k)\leq \frac{K_I(f)}{k}M(\Gamma_k^*)
$$
by Lemma~\ref{mcf3}.
By Lemma~\ref{mcf2} and~\eqref{capring} we have
$$
M(\Gamma_k^*)\leq M\!\left(\Delta\!\left(S(r),S(\theta r);B(\theta r)
\!\setminus\! \overline{B}(r)
\right)\right)=\omega_{n-1} (\log\theta)^{1-n}.
$$
Combining the last two inequalities we
obtain
\begin{equation}\label{p12a}
M(\Gamma_k)\leq \frac{K_I(f)}{k} \omega_{n-1} (\log\theta)^{1-n}.
\end{equation}
Applying Lemma~\ref{mcf4} with $F_1=F$, $F_2=B_k'$, $u=t/2$ and
$v=t$ 
and noting that
$M(\Delta(F,S(t);B(t)))=\capacity( B(t),F)$ 
by Lemma~\ref{mcf1}
we obtain
\begin{equation}\label{Gkc}
M(\Gamma_k)\geq 3^{-n}\min\left\{
\capacity( B(t),F), M(\Delta(B_k',S(t/2);B(t))),c_n\log 2\right\}.
\end{equation}
Let $k_0$ be the integer part of 
$$
\frac{ 3^n K_I(f) \omega_{n-1} (\log\theta)^{1-n}}{
\min\left\{
\capacity(B(t),F ), c_n\log 2\right\} } .
$$
Using~\eqref{p12a} we see that for 
$k>k_0$ 
the minimum in the right side of~\eqref{Gkc}
is attained by the term in the middle so that
$$
M(\Gamma_k)\geq  M(\Delta(B_k',S(t/2);B(t)))
$$
and thus 
$$
M(\Gamma_k)
\geq \alpha \left(\log \left(\frac{\beta t^{n-1}}{H^{n-1}(B_k')}
\right)\right)^{1-n}
$$
by Lemma~\ref{mcf5}. Together with~\eqref{p12a} we 
obtain
$$
\alpha \left(\log \left(\frac{\beta t^{n-1}}{H^{n-1}(B_k')}
\right)\right)^{1-n}
\leq 
\frac{K_I(f)}{k}
\omega_{n-1} (\log\theta)^{1-n} 
\quad\text{for}\ k> k_0.
$$
This yields
\begin{equation}\label{p12b}
\frac{H^{n-1}(B_k')}{t^{n-1}}\leq \beta
\exp\left(-\tau \log\theta
\left(\frac{k}{K_I(f)}\right)^{1/(n-1)}\right)
\quad\text{for}\ k> k_0,
\end{equation}
where $\tau=(\omega_{n-1}/\alpha)^{1/(1-n)}$ depends only on~$n$.
Now
\begin{equation}\label{p12d}
\begin{aligned}
\nu(r,S(z,t))
=&\frac{1}{\omega_{n-1}t^{n-1}}
\sum_{k=1}^{\infty} k H^{n-1}(A_k)\\
=&\frac{1}{\omega_{n-1}t^{n-1}}
\sum_{k=1}^{\infty} k \left(H^{n-1}(B_k)-H^{n-1}(B_{k+1})\right)\\
=&\frac{1}{\omega_{n-1}t^{n-1}}
\sum_{k=1}^{\infty} H^{n-1}(B_k)\\
\leq&
\frac{1}{\omega_{n-1}t^{n-1}}
\left( k_0 H^{n-1}(S(t))
+2\sum_{k=k_0+1}^{\infty} H^{n-1}(B_k') \right)\\
=& k_0 + \frac{2}{\omega_{n-1}t^{n-1}}
\sum_{k=k_0+1}^{\infty} H^{n-1}(B_k').
\end{aligned}
\end{equation}
By~\eqref{p12b} we have
\begin{equation}\label{p12e}
\begin{aligned}
\quad &
\frac{2}{\omega_{n-1}t^{n-1}}
\sum_{k=k_0+1}^{\infty} H^{n-1}(B_k')
\\
\leq&
\frac{2\beta}{\omega_{n-1}}
\sum_{k=1}^{\infty} 
\exp\left(-\tau \log\theta\left(\frac{k}{K_I(f)}\right)^{1/(n-1)}\right)\\
\leq&
\frac{2\beta}{\omega_{n-1}}
\int_{u=0}^{\infty}
\exp\left(-\tau \log\theta\left(\frac{u}{K_I(f)}\right)^{1/(n-1)}\right)du\\
=&
C_1\; K_I(f)(\log \theta)^{1-n},
\end{aligned}
\end{equation}
with 
$$
C_1:=\frac{2}{\omega_{n-1}\beta}
\int_{u=0}^{\infty}
\exp\left(-\tau u^{1/(n-1)}\right)du
$$
depending only on~$n$.
Noting that 
\begin{equation}\label{p12f}
\capacity(B(t),F)\leq \capacity\!\left(B(t),\overline{B}(t/2)\right)
= \omega_{n-1} (\log 2)^{1-n}
\end{equation}
by~\eqref{capring} we see that
\begin{equation}\label{p12g}
k_0
\leq C_2\frac{K_I(f)(\log\theta)^{1-n}}{\capacity\!\left(B(t),F\right)}
\end{equation}
for some constant $C_2$ depending only on~$n$.
Combining~\eqref{p12d}, \eqref{p12e} and~\eqref{p12g} we obtain
$$
\nu(r,S(z,t))\leq \left( C_1 + \frac{C_2}{\capacity\!\left(B(t),F\right)}
\right) K_I(f)(\log\theta)^{1-n},
$$
which together with~\eqref{p12f} yields~\eqref{nuSzt}
with $C:=\omega_{n-1} (\log 2)^{1-n}C_1+C_2$.\qed

\section{Proof of Theorems~\ref{thm1}--\ref{thm3}}\label{proofthm1}
\begin{proof}[Proof of Theorem \ref{thm1}]
First we note that
$$A\!\left(\RR,f^k\right)=\deg(f^k)=(\deg(f))^k$$
 for $k\in\N$.
Next we observe that there exists a constant $L$ depending 
only on $n$ such $\RR$ can be covered by $Lk^{n}$ 
balls of chordal radius $1/k$, for all $k\in\N$.
Hence for each $k\in\N$ there exists $x_{k}\in\RR$
such that
$$
A\!\left(\overline{B}_\chi\!\left(x_{k},1/k\right),f^k\right)
\geq \frac{1}{Lk^n} A\!\left(\RR,f^k\right)
=\frac{(\deg(f))^n}{Lk^n}.$$
The sequence $(x_k)$ has a convergent subsequence,
say $x_{k_j}\to x$.
We will show that~\eqref{1c} holds for every neighborhood $U$ of~$x$.

Suppose that this is not the case. 
Then there exists a set $F$ of positive capacity and
$\delta>0$ 
such that $O^+\!\left(\overline{B}_\chi(x,2 \delta)\right)
\subset \RR\setminus F$. 
Since $ K_I(f^k)\leq K_I(f)^k$ by~\eqref{KIfg}
we deduce from Theorem~\ref{thm-mr2}  and the remark
following it that
$$A\!\left(\overline{B}_\chi(x,\delta),f^k\right)
\leq C K_I(f)^k$$
for some constant~$C$.
On the other hand, for sufficiently large $k$ we have
$\overline{B}_\chi(x_{k_j},1/k_j)\subset \overline{B}_\chi(x,\delta)$ and thus
$$
A\!\left(\overline{B}_\chi(x,\delta),f^{k_j}\right)\geq
A\!\left(\overline{B}_\chi\left(x_{{k_j}},1/{k_j}\right),f^{k_j}\right)
\geq
\frac{(\deg(f))^{k_j}}{L k_j^n}.$$
The last two inequalities yield
$$\frac{(\deg(f))^{k_j}}{L k_j^n} \leq C K_I(f)^{k_j}.$$
For large $j$ this contradicts the assumption that $\deg(f)>K_I(f)$.
\end{proof}
\begin{proof}[Proof of Theorem \ref{thm0}]
Denote by $J_1(f)$ the set where the iterates are not normal
and by $J_2(f)$ the set given by Definition~\ref{defJ}.

If $x\in J_1(f)$ and $U$ is a neighborhood of~$x$,
 then, as already mentioned in the introduction,
$\RR\setminus O^+(U)$ is finite 
by Miniowitz's theorem and thus~\eqref{1c} holds. Hence $x\in J_2(f)$.

If $x\in F(f):=\RR\setminus J_1(f)$, then there exists
a neighborhood $U$ of~$x$ satisfying $U\subset F(f)$.
By the complete invariance of $F(f)$ and $J_1(f)$ we have 
$O^+(U)\subset F(f)$ and thus $\RR\setminus O^+(U)\supset J_1(f)$.
By the result of Fletcher and Nicks~\cite{FletcherNicks11}
already mentioned in the introduction, we have 
$\dim J_1(f)>0$ and thus $\dim\!\left( \RR\setminus O^+(U)\right)>0$.
Hence $\capacity\!\left( \RR\setminus O^+(U)\right)>0$
by Lemma~\ref{cap-dim}.
We conclude that $x\notin J_2(f)$. 

Altogether we see that $J_1(f)=J_2(f)$.
\end{proof}
\begin{proof}[Proof of Theorem \ref{thm3}]
It is easy to see that $J(f)\cap A(\xi)=\emptyset$.
Let $x\in J(f)$. 
Then $x\notin  A(\xi)$. Suppose now that 
 $x\notin  \partial A(\xi)$. Then there exists 
a neighborhood $U$ of $x$ such that $U\cap A(\xi)=\emptyset$
and thus $O^+(U)\cap A(\xi)=\emptyset$. 
In particular, $\RR\setminus O^+(U)\supset \overline{B}_\chi(\xi,\varepsilon)$
for some $\varepsilon>0$. Since $\overline{B}_\chi(\xi,\varepsilon)$ has
positive capacity, this is a contradiction.
\end{proof}

\section{Proof of Theorem~\ref{thm4}}\label{proofthm4}
The following result can be found in~\cite[Theorem~III.4.7]{Rickman93}.
\begin{lemma}\label{lemma-m70}
Let $\Omega\subset\R^n$ be open,
$f:\Omega\to\R^n$ be quasi\-regular and $x\in\Omega$.
Then there exists $A,B,r>0$ such that
\begin{equation}\label{locdist1}
A |x-y|^\nu \leq |f(x)-f(y)|\leq B|x-y|^\mu
\quad\text{for}\  
y\in B(x,r),
\end{equation}
 where
$\nu=(K_O(f)i(x,f))^{1/(n-1)}$ and
$\mu=(i(x,f)/K_I(f))^{1/(n-1)}$.
\end{lemma}
The right inequality of~\eqref{locdist1} is due to
Martio~\cite{Martio70}.
It was one of the main tools used by Fletcher and Nicks~\cite{FletcherNicks}
in their study of quasi\-regular maps of polynomial type.
The left inequality of~\eqref{locdist1} is due to Srebro~\cite{Srebro76}.
It will not be needed in the sequel and is listed here only for 
completeness.

Lemma~\ref{lemma-m70} extends to the case where the domain and range
of $f$ are in~$\RR$, provided the Euclidean metric is replaced
by the chordal metric.
Clearly the number $r$ in Lemma~\ref{lemma-m70} depends on~$x$, as
the left inequality of~\eqref{locdist1}
 implies that $f$ is injective in $B(x,r)$.
In the proof, the dependence on $r$ comes in because
$f$ is considered in a normal neighborhood of~$x$.
This is, by definition, a  neighborhood $U$ 
compactly contained in $\Omega$ such that
$f(\partial U) =\partial f(U)$ and $U\cap f^{-1}(f(x))=\{x\}$.

For a quasi\-regular map $f:\RR\to \RR$ we put 
$B_f^*:=\{x\in\RR: i(x,f)=\deg(f)\}$.
For $x\in B_f^*$ we have $f^{-1}(f(x))=\{x\}$
and thus the last condition in the definition of a 
normal neighborhood is automatically satisfied.
Using this observation
and noting that $B_f^*$ is compact we can deduce from the
proof of Lemma~\ref{lemma-m70} in~\cite{Rickman93} that we may take
the same values~$A$, $B$ and $r$ for all $x\in  B_f^*$.
Thus we obtain the following result.
\begin{lemma}\label{lemma-m70a}
Let $f:\RR\to \RR$ be quasi\-regular.
Then there exists $A,B,r>0$ such that 
if $x\in B_f^*$,
then 
\begin{equation}\label{locdist2}
A \, \chi(x,y)^\nu \leq \chi(f(x),f(y))\leq B \, \chi(x,y)^\mu
\quad\text{for}\  
y\in B_\chi(x,r),
\end{equation}
where $\nu=(K_O(f)\deg(f))^{1/(n-1)}$ and
$\mu=(\deg(f)/K_I(f))^{1/(n-1)}$.
\end{lemma}
We shall only need the right inequality of~\eqref{locdist2}.
This inequality can also be deduced from Corollary~\ref{cormar} below
which says that for $x$ in a compact subset of $\Omega$ one
may choose $B$ and $r$ in the right inequality of~\eqref{locdist1}
independent of~$x$.

\begin{proof}[Proof of Theorem \ref{thm4}]
Let  $x\in E(f)$. 
Since 
$\card f^{-1}(A)\geq \card A$
for every finite subset $A$ of $\RR$ and 
$ f^{-1}(O^-(x)\cup \{x\})=O^-(x)$
we have
$$\card O^-(x)
=\card (f^{-1}(O^-(x)\cup \{x\}))
\geq \card (O^-(x)\cup \{x\})$$
and thus $x\in O^-(x)$.
Hence $x$ is periodic.

Moreover, the argument shows that $\card f^{-1}(y)=1$ for all $y\in E(f)$.
Choosing $p\in\N$ such that $f^p(x)=x$ we thus have
$f^{-p}(x)=\{x\}$.
This implies that
$i(x,f^p)=\deg(f^p)$.
With $B$ and $r$ as in Lemma~\ref{lemma-m70a}, applied to $f^p$,
we obtain
$$\chi(f^p(y),x)=\chi(f^p(y),f^{p}(x))\leq B\,\chi(y,x)^\mu$$
for $\chi(y,x)< r$,
where 
$$\mu=
\left(\frac{\deg(f^p)}{K_I(f^p)}\right)^{1/(n-1)}.
$$
Since $\deg(f^p)=\deg(f)^p>K_I(f)^p\geq K_I(f^p)$
by~\eqref{KIfg} 
we have $\mu>1$.
Hence there exists $\delta>0$ depending only on~$B$, $r$ and $\mu$ such that
$$\chi(f^p(y),x)\leq\frac{1}{2} \chi(y,x)$$ 
for $\chi(y,x)\leq \delta$.
Thus 
$x$ is an attracting periodic point.
Moreover, we have $B_\chi(x,\delta)\subset A(x)$,
which implies that the chordal distance between two points in $E(f)$ 
is at least $\delta$. Thus  $E(f)$ is finite.
\end{proof}

\section{Hausdorff measure of invariant sets}\label{hausdorffinv}
Let  $\eta>0$.
 An increasing, continuous function 
$h:(0,\eta]\to (0,\infty)$ 
satisfying
$\lim_{t\to 0}h(t)=0$ 
is called  a \emph{gauge function}.
For a set $X\subset \R^n$ and a gauge function $h$
the \emph{Hausdorff measure} $H_h(X)$
is defined by
\[
H_h(X):=\lim_{\delta\to 0} \inf_{(X_j)}\sum^\infty_{j=1}\;h(\mbox{diam}\ X_j),
\]
where the infinimum is taken over all sequences $(X_j)$ 
of subsets of $\R^n$  such that
$X\subset \bigcup^\infty_{j=1}X_j$ and $\diam X_j<\delta$
for all~$j$.
The $d$-dimensional Hausdorff measure $H^d(X)$ considered already
corresponds to the function $h(t)=t^d$, up to a normalization factor.

Recall that for  $X\subset \R^n$ and 
$f:X\to \R^n$
an increasing, continuous function $\omega:[0,\infty)\to [0,\infty)$ is called 
a \emph{modulus of continuity} for $f$ if $\omega(0)=0$ and
$|f(x)-f(y)|\leq \omega(|x-y|)$ for all $x,y\in X$.
If
this holds with
$\omega(t)=L t^\alpha$ where $L,\alpha>0$,
then $f$
is said to be \emph{H\"older continuous} with exponent $\alpha$
and in the special case that $\alpha=1$ we say that $f$ is
\emph{Lipschitz continuous} with 
\emph{Lipschitz constant}~$L$. 
Identifying $\RR$ with $S^n(1)\subset\R^{n+1}$ we also use
this terminology for $X\subset \RR$ and $f:X\to\RR$.
(Equivalently, we can replace the Euclidean metric by the
chordal metric in the definition of the modulus of continuity.)
\begin{theorem} \label{Hhinv}
Let $X\subset \R^n$ be compact and let $f:X\to X$ be a continuous
function with modulus of continuity $\omega$.

Suppose that there exists $m\in\N$, $m\geq 2$, and $\delta>0$
such that each $y\in X$ has $m$ preimages $x_1,\dots,x_m$ satisfying
$|x_i-x_j|\geq\delta$ for $i\neq j$.
If $h$ is a gauge function such that
$\omega^k(h^{-1}(1/m^k))\leq \delta/2$ for all large~$k$,
then $H_h(X)>0$.
\end{theorem}
We give two corollaries dealing with Lipschitz and H\"older
continuous maps.
\begin{cor}\label{corHh1}
Let $X\subset \R^n$ be compact and let $f:X\to X$ be continuous
such that each $y\in X$ has $m$ preimages $x_1,\dots,x_m$ satisfying
$|x_i-x_j|\geq\delta$ for $i\neq j$.
If $f$ satisfies a Lipschitz condition with Lipschitz constant~$L>1$, then
$$\dim X\geq \frac{\log m}{\log L}.$$
\end{cor}
\begin{cor}\label{corHh2}
Let $X\subset \R^n$ be compact and let $f:X\to X$ be continuous
such that each $y\in X$ has $m$ preimages $x_1,\dots,x_m$ satisfying
$|x_i-x_j|\geq\delta$ for $i\neq j$.
If $f$ satisfies a H\"older condition with exponent $\alpha<1$,
then $H_h(X)>0$ for
$$
h(t)
=
\left(\log\frac{1}{t}\right)^{(\log m)/(\log \alpha)}.
$$
\end{cor}
For the proof of Theorem~\ref{Hhinv} we need the following 
version of the so-called mass distribution principle; 
see~\cite[Theorem~7.6.1]{Przytycki10}.
\begin{lemma} \label{la6}
Let $X\subset \R^n$ be compact and let 
$h$
be a gauge function.
Suppose that there exist a
probability measure $\mu$ supported on $X$ 
and $c,\eta>0$ 
such that 
$\mu(B(x,r))\leq c\,h(r)$ for $0<r\leq\eta$ and all $x\in X$.
Then $H_h(X) >0$.
\end{lemma}
\begin{proof}[Proof of Theorem~\ref{Hhinv}]
For each finite subset $E_0$ of $X$ we can choose a finite
subset $E_1$ of $f^{-1}(E_0)$ such that 
each point in $E_0$ has $m$ preimages in $E_1$, with
$|x-x'|\geq \delta$ if $x,x'\in E_1$ with $f(x)=f(x')$.
Clearly, $\card E_1= m\card E_0$.
Beginning with 
$E_0=\{y\}$ for some fixed $y\in X$
and performing this process repeatedly we obtain a sequence $(E_k)$ 
of sets with $\card E_k=m^k $ such that each point
in $E_{k-1}$ has $m$ preimages in $E_k$, with
$|x-x'|\geq \delta$ if $x,x'\in E_k$ with $f(x)=f(x')$.
We denote by $\delta_x$ the Dirac measure at a point $x$
and,  for $k\geq 0$, define the measure $\mu_k$ by
$$
\mu_k := \frac{1}{m^k}\sum_{x\in E_k } \delta_x.
$$
The sequence $(\mu_k)$ has a subsequence which
converges with respect to the weak$^*$-topology, say
$\mu_{k_j}\to \mu$; see, e.g., \cite[Theorem~6.5]{Walters}.

For $x\in X$ and $0<r\leq \delta/2$ we have
$$\mu_{k+1}(B(x,r))\leq \frac{1}{m}\mu_{k}(B(f(x),\omega(r))).$$
Thus
$$\mu_{k+l}(B(x,r))\leq \frac{1}{m^l}\mu_{k}(B(f^l(x),\omega^l(r)))$$
for $l\in\N$ as long as $\omega^{l-1}(r)\leq \delta/2$.

If $\omega^j(r)\leq \delta/2$ for all $j\in\N$, then
$\mu(B(x,r))=0$ for all $x\in X$ and thus
$\mu\equiv 0$, which is a contradiction.
Thus there exists $j\in\N$ depending on $r$  such that 
$\omega^{j-1}(r)\leq\delta/2 < \omega^{j}(r)$.
Denoting by $\tau$ the inverse function of $\omega$ we thus have
$\tau^{j}(\delta/2)<r\leq \tau^{j-1}(\delta/2)$.
Note that $j$ is large when $r$ is small.
By hypothesis we thus have
$\omega^{j}(h^{-1}(1/m^{j}))\leq \delta/2$
and hence $h(\tau^{j}(\delta/2))\geq 1/m^{j}$ for small~$r$,
say for $0<r\leq \eta$.
For $k\geq j$ we thus obtain
$$
\mu_{k}(B(x,r))
\leq \frac{1}{m^{j}}\mu_{k-j}(B(f^j(x),\omega^j(r)))
\leq \frac{1}{m^{j}}
\leq h(\tau^{j}(\delta/2))
\leq h(r).
$$
We conclude that 
$\mu(B(x,r))\leq  h(r)$ for $x\in X$ 
and $0<r\leq \eta$ so that the conclusion
follows from Lemma~\ref{la6}.
\end{proof}
\begin{proof}[Proof of Corollaries~\ref{corHh1} and~\ref{corHh2}]
Let $\omega(t)=L\,t^\alpha$. 
By induction we find that 
$$\omega^k(t)={L}^{p_k}t^{\alpha^k}
\quad\text{with}\ p_k=\sum_{j=0}^{k-1}\alpha^j$$
for $k\in\N$.

First we consider the case that $\alpha=1$. 
Then $p_k=k$ so that $\omega^k(t)=L^k t$.
Define
$h_1(t):=(2t/\delta)^{(\log m)/(\log L)}$.
Then 
$$
h_1^{-1}\!\left(\frac{1}{m^k}\right)
=\frac{\delta}{2}
\left(\frac{1}{m^k}\right)^{(\log L)/(\log m)}
=\frac{\delta}{2L^k} .$$ 
Hence $\omega^k(h_1^{-1}(1/m^k))=\delta/2$.
Thus $H_{h_1}(X)>0$ by Theorem~\ref{Hhinv}, and Corollary~\ref{corHh1}
follows.

Now we consider the case that $\alpha<1$.
Then $p_k=(1-\alpha^{k})/(1-\alpha)\sim 1/(1-\alpha)$
as $k\to\infty$ so that $\omega^k(t)\leq c\, t^{\alpha^k}$ for some
$c>0$.
With 
$b:= \log(2c/\delta)$
we now define 
$$
h_2(t)
:=
\left(\frac{1}{b} \log\frac{1}{t}\right)^{(\log m)/(\log \alpha)}.
$$
Then $h_2^{-1}(t)=\exp\!\left(-b\, t^{(\log  \alpha)/(\log m)}\right)$
and thus
$$
h_2^{-1}\!\left(\frac{1}{m^k}\right)
= \exp\!\left(-b \left(\frac{1}{m^k}\right)^{(\log  \alpha)/(\log m)}\right)
=\exp\!\left(-b  \alpha^{-k}\right).
$$
Hence $\omega^k(h_2^{-1}(1/m^k))\leq c\, e^{-b}= \delta/2$.
Thus $H_{h_2}(X)>0$ by Theorem~\ref{Hhinv}, and
Corollary~\ref{corHh2} follows.
\end{proof}

\section{Proof of Theorems~\ref{thm5}--\ref{thm8}}\label{proofthm5}
\begin{proof}[Proof of Theorem \ref{thm5}]
Suppose that 
there exists an open set $U$ intersecting $J(f)$ such
that $O^+(U)\not\supset \RR\setminus E(f)$.
Then there exists $z\in\RR\setminus (O^+(U)\cup E(f))$.
We note that
$O^-(z)$ is infinite since $z\notin E(f)$. Let $X$ be the set of 
limit points of $O^-(z)$. Then $X$ is a non-empty, closed
and completely invariant subset of $\RR\setminus O^+(U)$. Moreover,
$X\cap E(f)=\emptyset$ by Theorem~\ref{thm4}.

First we show that for each $x\in X$ there exist $m\in \N$
such that $f^{-m}(x)$ contains at least two points.
Otherwise there exist $y_0\in X$ such that 
for all $k\in\N$ we have
$f^{-k}(y_0)=\{y_k\}$
for some $y_k\in X$.
With $B_f^*$ as in Lemma~\ref{lemma-m70a} we find that
$y_k\in B_f^*$ for all $k\in\N$.
It follows from Lemma~\ref{lemma-m70a} that 
there exists $\delta>0$ such that
if $x\in B_f^*$ and $y\in\RR$ with $\chi(x,y)<\delta$,
then $\chi(f(x),f(y))\leq \chi(x,y)/2<\delta/2$.
Now there exists $N\in \N$ such that if $x_1,\dots,x_N\in \RR$, 
then there exist $k,l\in\{1,\dots,N\}$ with $k\neq l$
such that $\chi(x_k,x_l)<\delta$. 
In particular, for each $m\in \N$ there exist
$k,l\in\{1,\dots,N\}$ with $k\neq l$  such that
$\chi(y_{m+k},y_{m+l})<\delta$.
It follows that 
\begin{equation}\label{p5a}
\chi(y_k,y_l)=\chi(f^m(y_{m+k}),f^m(y_{m+l}))
\leq \frac{1}{2^m}\chi(y_{m+k},y_{m+l})<\frac{1}{2^m}\delta.
\end{equation}
On the other hand,  since $y_0\notin E(f)$, all the points
$y_k$ are distinct and thus 
\begin{equation}\label{p5b}
\eta:=\min_{1\leq k<l\leq N}\chi(y_k,y_l)>0.
\end{equation}
Choosing $m$ such that $2^m>\delta/\eta$ we obtain a contradiction
from~\eqref{p5a} and~\eqref{p5b}.
Thus  for each $x\in X$ there exist $m\in \N$
such that $f^{-m}(x)$ contains at least two points.

Noting that  $f$ is an open map we deduce that for every $x\in X$ 
there exist $m(x)\in \N$, $\delta(x)>0$ and 
a neighborhood $U(x)$ of $x$ such that if $y\in U(x)$, then
$f^{-m(x)}(y)$ contains two points whose chordal distance
is at least $\delta(x)$. The compact set $X$ can be
covered by finitely many such neighborhoods, say
$X\subset \bigcup_{j=1}^k U(x_j)$.
Let $m:=\max_j m(x_j)$. 
Since $f$ and its iterates are continuous 
there exists $\delta>0$ such that
$\chi(f^{m-m(x_j)}(x),f^{m-m(x_j)}(y))<\delta(x_j)$ for 
$x,y\in X$ with $\chi(x,y)<\delta$ and $j\in\{1,\dots,k\}$.
We find that for each $x\in X$ the preimage $f^{-m}(x)$ 
contains two points whose chordal distance
is at least~$\delta$.

It now follows from Corollary~\ref{corHh1}, which we may apply 
by considering the subset $X$ of $\RR$ as a subset of $\R^{n+1}$,
that $\dim X>0$.
Thus $\capacity X>0$ by Lemma~\ref{cap-dim}.
On the other hand, since $X\subset \RR\setminus O^+(U)$, we 
have $\capacity X=0$ by Theorem~\ref{thm1}.
This is a contradiction.

Hence 
$O^+(U)\supset \RR\setminus E(f)$ if  $U$ is 
an open set intersecting $J(f)$.
Since $J(f)$ is completely invariant and 
$J(f)\cap E(f)=\emptyset$ by Theorem~\ref{thm4}, we also
deduce that 
$O^+(U\cap J(f))=J(f)$.
\end{proof}
\begin{proof}[Proof of Theorem \ref{thm7}]
Let $x\in \RR\setminus E(f)$ and suppose that 
$J(f)\not \subset \overline{O^-(x)}$ so that there exists
$y\in J(f)\setminus \overline{O^-(x)}$. Then 
$y$ has a neighborhood $U$ 
satisfying
$U\cap \overline{O^-(x)}=\emptyset$.
It follows that $O^+(U)\cap \overline{O^-(x)}=\emptyset$.
Thus $O^-(x)\subset E(f)$ by Theorem~\ref{thm5}.
Hence $x\in E(f)$, which 
is a contradiction.
We deduce that $J(f)\subset \overline{O^-(x)}$
for $x\in \RR\setminus E(f)$.

Since $J(f)\cap E(f)=\emptyset$ by Theorem~\ref{thm4} we 
deduce that $J(f)\subset \overline{O^-(x)}$ holds in particular
for $x\in J(f)$. On the other hand, if 
$x\in J(f)$, then $O^-(x)\subset J(f)$ since $J(f)$ is 
completely invariant and hence $\overline{O^-(x)}\subset J(f)$ 
since $J(f)$ is closed.
It follows that $\overline{O^-(x)}= J(f)$ for $x\in J(f)$.
\end{proof}
\begin{proof}[Proof of Theorem \ref{thm6}]
We use the same method 
as in the proof of Theorem~\ref{thm5}.
In fact, here the argument is even a little simpler.

Noting again that $J(f)\cap E(f)=\emptyset$ by Theorem~\ref{thm4} 
we see as in the proof of Theorem~\ref{thm5} 
that there exist
$m\in\N$ and $\delta>0$ such that for each $x\in J(f)$ the
preimage $f^{-m}(x)$ contains two points whose chordal
distance is at least~$\delta$.
Corollary~\ref{corHh1} now yields the conclusion.
\end{proof}
The proof of Theorem~\ref{thm8} requires
the following strengthening
 of Lemma~\ref{cap-dim}; see~\cite{Wallin77}.
\begin{lemma}\label{cap-haus}
Let $X\subset\R^n$ be compact and $\eps>0$.
If
$H_h(X)>0$
for
\begin{equation} \label{chh}
h(t) =
\left(\log\frac{1}{t}\right)^{1-n-\eps},
\end{equation}
then $\capacity X>0$.
\end{lemma}
\begin{proof}[Proof of Theorem \ref{thm8}]
The argument is similar to that used in the proof of Theorem~\ref{thm6} 
(and Theorem~\ref{thm5}). However, since we assume that 
$J(f)\cap B_f=\emptyset$, we now find that there exists
$\delta>0$ such  that each $x\in J(f)$ has $d:=\deg(f)$ preimages,
any two of which have chordal distance at least~$\delta$.
Moreover, $f$ is H\"older continuous with exponent
$\alpha:=K_I(f)^{1/(1-n)}$;  see~\cite[Theorem III.1.11]{Rickman93}
or
Corollary~\ref{cormar} below.
Corollary~\ref{corHh2} now yields that 
$H_h(J(f))>0$ for
$$
h(t)
=
\left(\log\frac{1}{t}\right)^{(\log d)/(\log \alpha)}
=
\left(\log\frac{1}{t}\right)^{((1-n) \log d)/(\log K_I(f))} .
$$
Since we assume that $d>K_I(f)$ we have 
$$
\frac{(1-n)\log d}{\log K_I(f)}=1-n-\eps 
$$
for some $\eps>0$. Now  the conclusion follows from
Lemma~\ref{cap-haus}.
\end{proof}

\section{Local distortion of quasi\-regular maps}\label{locdist}
The following lemma gives a generalization of
the right inequality of~\eqref{locdist1} which
may be of independent interest.

First we introduce some terminology.
Let $f:\Omega\to\R^n$ be quasi\-regular and let $U$ be 
a domain compactly contained in $\Omega$.
Then $U$ is called a \emph{normal domain} for $f$
if $f(\partial U) =\partial f(U)$.
The normal neighborhood of a point $x$ already mentioned in
section~\ref{proofthm4} is thus a normal domain $U$
satisfying $U\cap f^{-1}(f(x))=\{x\}$.
For $x\in \Omega$ and $s>0$ we denote by $U(x,f,s)$
the component of $f^{-1}(B(f(x),s))$ that contains~$x$
and by $\overline{U}(x,f,s)$ its closure.
We note that if 
$U(x,f,s)$ is compactly contained
in $\Omega$, then $U(x,f,s)$ is a normal domain
and thus $f$ is a proper map from $U(x,f,s)$
onto $B(f(x),s)$. We denote the degree of this map
by $d(x,f,s)$.
\begin{lemma}\label{mcf6}
For  $M,n\in \N$ and $K\geq 1$ there exists $c,\eta>0$ with
the following properties:
let $f:\Omega\to\R^n$ be $K$-quasi\-regular, $x\in\Omega$,
$\sigma>0$ and $0<s\leq \eta\sigma$.
Suppose that $U(x,f,\sigma)$ is compactly contained in $\Omega$
and  
$d(x,f,\sigma)\leq M$.
Define
\[
\mu(m):=\left(\frac{m}{K_I(f)}\right)^{1/(n-1)}
\]
for $1\leq m\leq M$.
Then 
there exists an integer~$m$,
depending on $x$ and $s$ and
satisfying
$d(x,f,s)\leq m\leq M$,
such that
\begin{equation}\label{martio2}
U(x,f,s/\eta)\supset \overline{B}\!\left(x,c\,s^{1/\mu(m)}\right)
\end{equation}
and
\begin{equation}\label{martio2a}
n\!\left( \overline{B}\!\left(x,c\,s^{1/\mu(m)}\right),y\right)\leq m
\quad \text{for}\ y\in B(f(x),s).
\end{equation}
\end{lemma}
One consequence of this lemma is the following result which says
that for $x$ in a compact subset of $\Omega$ the constants $B$ and $r$ in 
the right inequality of~\eqref{locdist1} can be chosen
independently of~$x$.
(This result is probably known, but I have not been able to 
find it in the literature.)
\begin{cor}\label{cormar}
Let 
$f:\Omega\to\R^n$ be quasi\-regular and $X\subset\Omega$ compact.
Then there exists $C,r>0$ such that
$$
|f(x)-f(y)|\leq C|x-y|^{\mu(i(x,f))}
\quad\text{for}\  
x\in X
\ \text{and}\ y\in B(x,r),
$$
 with $\mu(m)$ defined as in Lemma~\ref{martio2}.
\end{cor}
\begin{proof}
Since $X$ is compact, there exist $M\in \N$ and  $\sigma>0$ such that 
$U(x,f,\sigma)$ is compactly contained in $\Omega$
and  $d(x,f,\sigma)\leq M$
for all $x\in X$.
Let $0<s\leq\eta\sigma$
and choose $m$ according to Lemma~\ref{martio2}.
Since $m\geq d(x,f,s)\geq i(x,f)$ 
we have $s^{1/\mu(m)}\geq s^{1/\mu(i(x,f))}$ and hence
we deduce from~\eqref{martio2} that
$$U(x,f,s/\eta)\supset B\!\left(x,c\,s^{1/\mu(i(x,f))}\right).$$
Thus $|f(y)-f(x)|\leq s/\eta$ if $|y-x|= c\,s^{1/\mu(i(x,f))}$.
Solving the last equation for $s$ and substituting the result
into the estimate for $|f(y)-f(x)|$
we obtain the conclusion.
\end{proof}

The proof of Lemma~\ref{mcf6} requires the following result
known as the $K_O$-inequality~\cite[Theorem~II.10.9]{Rickman93}.
Here a condenser $(G,C)$ is called \emph{normal} for a quasi\-regular
map $f$ if $G$ is a normal domain for~$f$.
\begin{lemma}\label{mcfKO}
Let $(G,C)$ be a normal  condenser for a quasi\-regular map~$f$.
If $N\in\N$ is such that $\card( f^{-1}(y)\cap G)\leq N$
for all $y\in f(G)$, then
$$\capacity (G,C)\leq N\,K_O(f) \capacity (f(G),f(C)).$$
\end{lemma}

We shall also need  the following lemma which can be 
deduced from~\cite[Lemma~5.42]{Vuorinen88}.
Here a condenser $(G,C)$ is called
\emph{ringlike} if $C$ and $\R^n\setminus G$ are connected.
\begin{lemma}\label{lemmax}
There exists 
$\kappa>0$
depending only on the dimension
such that if
$(G,C)$ is a ringlike condenser with $\capacity (G,C)<\kappa$, then
$\overline{B}(x,\diam C)\subset G$ for all $x\in C$.
\end{lemma}

\begin{proof}[Proof of Lemma~\ref{mcf6}]
With 
$\rho:=\sup\{r>0: B(x,r)\subset U(x,f,\sigma)\}$
we have
$f(B(x,\rho))\subset B(f(x),\sigma)$.
Put $\eta:=1/L^{M+1}$ with a constant $L>1$ to be determined later.
Clearly, the function $t\mapsto d(f,x,t)$ is non-decreasing and takes
values in $\{1,2,\dots,M\}$. 
It follows that 
for
$0<s\leq \eta\sigma=\sigma/L^{M+1}$ there exists $t(s)\in [s,L^M s]$
such that $t\mapsto d(x,f,t)$ is constant in $[t(s),Lt(s)]$.
We put $m:=d(x,f,t(s))$.
From Lemma~\ref{mcfKO} we deduce that
$$
\begin{aligned}
\capacity\!\left( U(x,f,Lt(s)),\overline{U}(x,f,t(s))\right)
&\leq m K_O(f) \capacity\!\left(B(x,Lt(s)),\overline{B}(x,t(s))\right)\\
&=m K_O(f)\omega_{n-1} (\log L)^{1-n} .
\end{aligned}
$$
Since $m\leq M$ we see that that if $L$ is chosen large,
then the right hand side
is less than the constant $\kappa$ from Lemma~\ref{lemmax}.
Denoting by $\tau(s)$ the diameter of $U(x,f,t(s))$
we conclude that
\begin{equation}\label{tau1}
\overline{B}(x,\tau(s))\subset U(x,f,Lt(s))\subset U(x,f,L^{m+1}s)
=U(x,f,s/\eta).
\end{equation}
In particular, 
$\overline{B}(x,\tau(s)) \subset U(x,f,\sigma)$ and thus
$\tau(s)<\rho$.

By Lemma~\ref{mcf3a} we have
\[
\capacity\!\left( B(f(x),\sigma), \overline{B}(f(x),t(s))\right)
\leq \frac{K_I(f)}{m}\capacity\!\left( U(x,f,\sigma),
\overline{U}(x,f,t(s))\right) .
\]
Now 
$B(x,\rho)\subset U(x,f,\sigma)$
and $\overline{U}(x,f,t(s))\subset \overline{B}(x,\tau(s))$.
Noting that $\tau(s)<\rho$ we obtain
\[
\capacity\!\left( B(f(x),\sigma), \overline{B}(f(x),t(s))\right)
\leq \frac{K_I(f)}{m}\capacity\!\left( B(x,\rho),
\overline{B}(x,\tau(s))\right),
\]
Using \eqref{capring} we deduce that
$$
\left(\log\frac{\rho}{\tau(s)}\right)^{n-1}
\leq \frac{K_I(f)}{m}
\left(\log\frac{\sigma}{t(s)}\right)^{n-1} .
$$
Solving this inequality for $\tau(s)$ we obtain
$$
\tau(s)
\geq
\rho \left(\frac{t(s)}{\sigma}\right)^{1/\mu(m)} 
\geq c\, s^{1/\mu(m)}
$$
for some positive constant~$c$.
Together with~\eqref{tau1} this yields~\eqref{martio2}.

Since $\overline{B}(x,\tau(s))\subset U(x,f,Lt(s))$ by~\eqref{tau1}
we have
$n\!\left(\overline{B}(x,\tau(s)),y\right)\leq m$
for $y\in B(f(x),Lt(s))$ and thus, in particular, for $y\in B(f(x),s)$.
This is~\eqref{martio2a}.
\end{proof}

\section{A conjecture about the capacity of invariant sets}\label{conj}
We conjectured in the introduction that the 
hypothesis that $f$ is Lipschitz continuous
can be omitted in Theorem~\ref{thm5} and~\ref{thm7}.
The proof of these theorems shows that this conjecture would follow
from the next one.
\begin{conjecture}\label{conjecture-cap}
Let $f:\RR\to\RR$ be quasi\-regular with $\deg(f)>K_I(f)$.
Let $X\subset \RR\setminus E(f)$ be compact and completely invariant.
Then $\capacity X>0$.
\end{conjecture}
While we have been unable to prove this conjecture, we 
give some arguments in favor of it.
The idea is to consider, for fixed $y\in X$, the measures
\begin{equation}\label{c11}
\nu_k
:=
\frac{1}{\deg(f)^k}\sum_{x\in f^{-k}(y)} i\!\left(x,f^k\right) \delta_x .
\end{equation}
These measures play an important role in
complex dynamics; cf.~\cite[Section~161]{Steinmetz93}.
Let $\nu$ be the limit of a convergent subsequence.
Assuming without loss of generality that $X\subset\R^n$,
we deduce from Lemmas~\ref{la6} and~\ref{cap-haus} that
 it suffices to prove
that $\nu(B(x,r))\leq h(r)$ for $0<r\leq\eta$ and all $x\in X$.

Now it follows from Lemma~\ref{mcf6} that if $s$ is sufficiently small
and $x\in X$,
then there exists $m$ satisfying $i(x,f)\leq m\leq \deg(f)$ such that
$$\nu\!\left(B\!\left(x,c s^{1/\mu(m)} \right)\right) \leq \frac{m}{\deg(f)}
\nu(B(f(x),s/ \eta )).$$
Thus given $\tau>1$ there exists $\rho_0\in (0,1)$ such that 
 if $x\in X$ and $0<r\leq \rho_0$,
then
\begin{equation}\label{d11}
\nu\!\left(B\!\left(x,r^{\tau/\mu(m)} \right)\right) \leq \frac{m}{\deg(f)}
\nu(B(f(x),r))
\end{equation}
for some $m=m(x,r)$ satisfying $i(x,f)\leq m\leq \deg(f)$ .

Let $x_0\in X$ and put $x_k:=f^k(x_0)$ for $k\in \N$.
Fix $k\in\N$, put $t_{k,k}:=\rho_0$ and,
for $0\leq j\leq k-1$,
define $t_{j,k}$ recursively by 
$t_{j,k}:=t_{j+1,k}^{\tau/\mu(m_j)}$
where $m_j:= m(x_j,t_{j+1,k})$.
Finally put $\rho_k:=t_{0,k}$.
Suppose for simplicity that $t_{j,k}\leq \rho_0$ for all~$j$.
It then follows from~\eqref{d11} that
$$\nu(B(x_0,\rho_k))\leq 
\left(\prod_{j=0}^{k-1} \frac{m_j}{\deg(f)}\right)
\nu(B(x_k,\rho_0))\leq \frac{1}{\deg(f)^k} \prod_{j=0}^{k-1} m_j.$$
Since 
$$\frac{\log \rho_k}{\log \rho_0}
=\prod_{j=0}^{k-1}\frac{\tau}{\mu(m_j)}
=\left(\tau K_I(f)^{1/(n-1)}\right)^k \prod_{j=0}^{k-1}
\left(\frac{1}{m_j}\right)^{1/(n-1)}$$
we have 
$$
h(\rho_k)
=
h(\rho_0)
\left(\frac{1}{\tau^{n+\eps-1} K_I(f)^{(n+\eps-1)/(n-1)}}\right)^{\! k}
\; \prod_{j=0}^{k-1}
m_j^{(n+\eps-1)/(n-1)}
$$
for the function $h$ defined by~\eqref{chh}.
Choosing $\tau$ close to $1$ and $\eps$ small we 
have $\tau^{n+\eps-1} K_I(f)^{(n+\eps-1)/(n-1)}<\deg (f)$ and hence
$\nu(B(x_0,\rho_k))\leq h(\rho_k)$ for large~$k$.

However, in order to apply Lemma~\ref{cap-haus} we would need
that $\nu(B(x_0,r))\leq h(r)$ for all small~$r$, not only
on a sequene of $r$-values.
Therefore this argument  can only be considered as
support for  Conjecture~\ref{conjecture-cap},
it does not prove it.
\begin{rem}
With the terminology of Lemma~\ref{mcf6}, let $r=c\,s^{1/\mu(m)}$.
It then follows from~\eqref{martio2} and~\eqref{martio2a}
that 
\begin{equation}\label{martio3}
f\!\left( \overline{B}(x,r)\right)
\subset  \overline{B}\!\left(f(x),C r^{\mu(m)}\right)
\end{equation}
and
\begin{equation}\label{martio3a}
n\!\left( \overline{B}(x,r),y\right)\leq m
\quad \text{for}\ y\in \overline{B}\!\left(f(x),C r^{\mu(m)}\right)
\end{equation}
for some $C>0$.
However, Lemma~\ref{mcf6} only says that for every $s$ there 
exists $m$ such that~\eqref{martio3} and~\eqref{martio3a} hold
with $r$ as defined above.
It does not yield that for every $r$ there 
exists $m$ such that~\eqref{martio3} and~\eqref{martio3a} hold.
Under the assumption  that this is the case we can prove
Conjecture~\ref{conjecture-cap}.
\end{rem}
\begin{theorem}\label{thm-conj-cap}
Let $f:\RR\to\RR$ be quasi\-regular with $\deg(f)>K_I(f)$.
Let $X\subset \RR\setminus E(f)$ be compact and completely invariant.
Suppose that there exist $C,\rho>0$ such that for every $x\in X$ and 
$r\in(0,\rho]$
there exists $m\in \{1,\dots,\deg(f)\}$ such that~\eqref{martio3} 
and~\eqref{martio3a} hold.
Then $\capacity X>0$.
\end{theorem}
\begin{proof}
Fix $y\in X$ and let $\nu$ be limit of a convergent subsequence
of the sequence $(\nu_k)$ defined by~\eqref{c11}.
Using~\eqref{martio3} and~\eqref{martio3a} we find that 
for $0<r\leq \rho$ and $x\in X$ there exists $m=m(x,r)$ such that
$$
\nu(B(x,r)) \leq \frac{m}{\deg(f)}
\nu\!\left(B\!\left(f(x),Cr^{\mu(m)}\right)\right).
$$
Given $\alpha\in (0,1)$ there thus exists $\rho_0>0$ such that 
$$
\nu(B(x,r)) \leq \frac{m}{\deg(f)}
\nu\!\left(B\!\left(f(x),r^{\alpha\mu(m)}\right)\right)
$$
for $0<r\leq \rho_0$.
Now  let $x\in X$ and $0<r\leq \rho_0$ and
put $x_0:=x$, $r_0:=r$ and $m_0:=m(x_0,r_0)$, 
and define
$x_k$, $r_k$ and $m_k$ for $k\geq 1$ recursively by 
$x_k:=f(x_{k-1})$,
$r_k:=r_{k-1}^{\alpha\mu(m_{k-1})}$
and 
$m_k:=m(x_k,r_k)$, as long as $r_k\leq \rho_0$.

First suppose that the process stops after $k$ steps; 
that is, $r_{k}^{\alpha\mu(m_{k})}>\rho_0$.
Then $r_{k}>\rho_1:=\rho_0^{1/(\alpha\mu(m_{1}))}$.
We conclude that
\begin{equation}\label{p11a}
\frac{\log r}{\log \rho_1}
\leq 
\frac{\log r_0}{\log r_k}
=\prod_{j=0}^{k-1} \alpha \mu(m_{j})
=\left( \frac{\alpha }{K_I(f)^{1/(n-1)}}\right)^{\! k}
\; \prod_{j=0}^{k-1} m_{j}^{1/(n-1)}.
\end{equation}
On the other hand,
\begin{equation}\label{p11b}
\begin{aligned}
\nu(
B(x,r))
&=\nu(B(x_0,r_0))
\\ &
\leq 
\left(\prod_{j=0}^{k-1} \frac{m_j}{\deg(f)}\right) 
\nu(B(x_k,r_k))
\\ &
\leq \frac{1}{\deg(f)^k} \prod_{j=0}^{k-1} m_{j}.
\end{aligned}
\end{equation}
Choosing $\alpha$ close to $1$ we deduce from~\eqref{p11a}
and~\eqref{p11b}
that there exist $c,\eps>0$ such that
\begin{equation}\label{p11c}
\nu(B(x,r))\leq c 
\left(\log\frac{1}{r}\right)^{1-n-\eps}.
\end{equation}

Suppose now that the inductive process defining $x_k$, $r_k$ 
and $m_k$ does not stop; that is, $r_k\leq \rho_0$ for all
$k\in\N$.
We shall show that there are infinitely many~$k$ such that
$m_k<\deg(f)$. In order to do so, we assume that this 
is not the case; say $m_k=\deg(f)$ for $k\geq k_0$.
With $B_f^*$ as defined in Lemma~\ref{lemma-m70a} we then
have $x_k\in B_f^*$ for $k\geq k_0$ and it follows from
this lemma that there exists $\delta>0$ such that if 
$|x-x_k|\leq \delta$, then $|f(x)-f(x_k)|\leq |x-x_k|/2$,
provided $k\geq k_0$.
Now there exist $l\geq k_0$ and $p\in\N$ such that 
$|x_{l+p}-x_l|\leq \delta$.
We deduce that $|x_{l+(j+1)p}-x_{l+jp}|\leq \delta/2^{jp}$.
Thus $(x_{l+jp})_{j\in\N}$ is a Cauchy sequence and
hence convergent; say $x_{l+jp}\to\xi$.
It follows that $f^p(\xi)=\xi$ and $\xi\in B_f^*$.
This implies that $\xi\in E(f)$. 
On the other hand, we have $\xi\in X$ since $X$ is compact.
Since $X\cap E(f)=\emptyset$ by hypothesis, this is a 
contradiction.

Thus $m_k<\deg(f)$ for infinitely many~$k$.
It now follows from~\eqref{p11b} that $\nu(B(x,r))=0$.
Thus~\eqref{p11c} also holds in this case.
The conclusion now follows from~\eqref{p11c}, Lemma~\ref{la6} and
Lemma~\ref{cap-haus}.
\end{proof}
\begin{rem}
We have restricted to sets $X\subset \R^n$ in Theorem~\ref{thm-conj-cap}
only for simplicity.
It also holds for $X\subset \RR$, provided we replace
the Euclidean balls in~\eqref{martio3} and~\eqref{martio3a}
by balls with respect to the chordal metric.

It follows from our considerations that if $f:\RR\to\RR$
is quasi\-regular with $\deg(f)>K_I(f)$ and 
if~\eqref{martio3} and~\eqref{martio3a} hold (with chordal
balls) for all $x\in\RR$ and $r\in (0, \rho_0]$ with some
$m=m(x,r)\in\{1,\dots,\deg(f)\}$, then
the conclusions of Theorems~\ref{thm5} and~\ref{thm7} hold.
\end{rem}


\begin{thebibliography}{99}
\bibitem{Ahlfors}
Lars V.\ Ahlfors, 
Conformal Invariants:
Topics in Geometric Function Theory.
McGraw-Hill, New York, D\"usseldorf, Johannesburg, 1973.
\bibitem{Beardon91} Alan F.\ Beardon,
{\rm Iteration of rational functions}.
Graduate Texts in Mathematics 91.
Springer-Verlag, New York, 
1991.
\bibitem{Bergweiler10}
Walter Bergweiler,
Iteration of quasi\-regular mappings
Comput.\ Methods Funct.\ Theory 10 (2010),
455--481.
\bibitem{Bergw}
Walter Bergweiler,
Karpi\'nska's paradox in dimension~$3$.
Duke Math.\ J.\ 154 (2010), 599--630.
\bibitem{BergwErem}
Walter Bergweiler and Alexandre Eremenko,
Dynamics of a higher dimensional analogue of the trigonometric
functions.
Ann.\ Acad.\ Sci.\ Fenn.\ Math.
36 (2011), 165--175.
\bibitem{Bergweiler09}
Walter Bergweiler, Alastair Fletcher, Jim Langley and Janis Meyer,
The escaping set of a quasi\-regular mapping.
Proc.\ Amer.\ Math.\ Soc.
137  (2009), 641--651.
\bibitem{Eremenko89}
A.\ \`E.\ Erem\"enko,
On the iteration of entire functions.
In ``Dynamical systems and ergodic theory''.
Banach Center Publications~23. 
Polish Scientific Publishers, Warsaw 1989, 339--345.
\bibitem{FletcherNicks}
Alastair Fletcher and Daniel A.\ Nicks,
Quasiregular dynamics on the $n$-sphere.
{\rm Ergodic Theory Dy\-nam.\ Systems} 31 (2011), 23--31.
\bibitem{FletcherNicks11}
Alastair Fletcher and Daniel A.\ Nicks,
Julia sets of uniformly quasi\-regular mappings are uniformly perfect.
Preprint, arXiv: 1012.1378.
\bibitem{Garber}
V. Garber,
 On the iteration of rational functions.
 Math. Proc. Cambridge Philos. Soc. 84 (1978), 497--505.
\bibitem{Gehring76}
F. W. Gehring, 
A remark on domains quasiconformally equivalent to a ball.
Ann. Acad. Sci. Fenn. Ser. A I Math. 2 (1976), 147--155.
\bibitem{Hinkkanen04}
Aimo Hinkkanen, Gaven J.\ Martin and Volker Mayer,
Local dynamics of uniformly quasi\-regular mappings.
Math. Scand.  95  (2004),  80--100.
\bibitem{Iwaniec01}
Tadeusz Iwaniec and Gaven J.\ Martin,
Geometric Function Theory and Non-linear Analysis.
Oxford Mathematical Monographs.
Oxford University Press, New York, 2001.
\bibitem{Martin97}
Gaven J. Martin,
Branch sets of uniformly quasi\-regular maps.
Conform. Geom. Dyn.  1  (1997), 24--27.
\bibitem{Martio70}
Olli Martio,
A capacity inequality for quasi\-regular mappings.  
Ann. Acad. Sci. Fenn. Ser. A I, No. 474 (1970), 18 pp. 
\bibitem{Mattila79}
P. Mattila and S. Rickman, 
Averages of the counting function of a quasi\-regular mapping. 
Acta Math. 143 (1979), 273--305.
\bibitem{Mayer97}
Volker Mayer,
Uniformly quasi\-regular mappings of Latt\`es type.
Conform. Geom. Dyn.  1  (1997), 104--111.
\bibitem{Milnor06} John Milnor,
Dynamics in one complex variable.
Third edition. Annals of Mathematics Studies 160.
Princeton University Press, Princeton, NJ, 2006.
\bibitem{Miniowitz82}
Ruth Miniowitz,
 Normal families of quasi\-meromorphic mappings.
 Proc.\ Amer.\ Math.\ Soc.\ 84 (1982), 35--43.
\bibitem{Przytycki10}
F.\ Przytycki and M.\ Urba\~nski,  
Conformal Fractals: Ergodic Theory Methods.
London Mathematical Society Lecture Note Series 371.
 Cambridge
University Press, Cambridge, 2010.

\bibitem{Reshetnyak}
Yu.\ G.\ Reshetnyak, 
Space mappings with bounded distortion.
Translations of Mathe\-matical Monographs 73. 
Amer.\ Math.\ Soc., Providence, RI, 1989. 
\bibitem{Rickman80}
Seppo Rickman,
On the number of omitted values of entire quasi\-regular mappings.
J.~Analyse Math. 37 (1980), 100--117.
\bibitem{Rickman85}
Seppo Rickman,
The analogue of Picard's theorem for quasi\-regular mappings in dimension three.
Acta Math.
154 (1985), 195--242.
\bibitem{Rickman93}
Seppo Rickman,
Quasiregular Mappings.
Ergebnisse der Mathematik und ihrer Grenzgebiete (3) 26.
Springer-Verlag, Berlin, 1993.
\bibitem{Sarvas72}
Jukka Sarvas, 
Symmetrization of condensers in $n$-space. Ann. Acad. Sci. Fenn. Ser. A I,
No. 522 (1972), 44 pp. 
\bibitem{Siebert04}
Heike Siebert,
Fixpunkte und normale Familien quasi\-regul\"arer Abbildungen.
Dissertation, University of Kiel,
 2004; http://e-diss.uni-kiel.de/diss\_1260.
\bibitem{Srebro76}
Uri Srebro, Quasiregular mappings.
In ``Advances in Complex Function Theory'',
Lecture Notes in Mathematics 505, Springer-Verlag, Berlin, 1976, pp. 148--163. 
\bibitem{Steinmetz93}
Norbert Steinmetz,
{\rm Rational Iteration}.
De Gruyter Studies in Mathematics 16.
Walter de Gruyter \& Co., Berlin
1993.
\bibitem{SunYang99}
Daochun Sun and Lo Yang,
Quasirational dynamical systems (Chinese).
 Chinese Ann. Math. Ser. A  20  (1999),  673--684.
\bibitem{SunYang00}
Daochun Sun and Lo Yang,
Quasirational dynamic system.
Chinese Science Bull. 45 (2000), 1277--1279.
\bibitem{SunYang01}
Daochun Sun and Lo Yang,
Iteration of quasi-rational mapping.  
Progr. Natur. Sci. (English Ed.)  11  (2001),  16--25.
\bibitem{Vuorinen83}
Matti Vuorinen, 
Some inequalities for the moduli of curve families.
Michigan Math.~J. 30 (1983), 369--380.
\bibitem{Vuorinen88}
Matti Vuorinen, 
Conformal Geometry and Quasiregular Mappings, 
Lecture Notes in Mathema\-tics 1319, Springer-Verlag, Berlin, 1988. 
\bibitem{Wallin77}
Hans Wallin, Metrical characterization of conformal capacity zero.
J.\ Math.\ Anal.\ Appl.\ 58 (1977), 298--311.
\bibitem{Walters}
Peter Walters,
An Introduction to Ergodic Theory.
Graduate Texts in Mathematics 79.
Springer-Verlag, New York, 1982.
\bibitem{Zalcman75}
Lawrence Zalcman,
A heuristic principle in complex function theory.
Amer.\ Math.\ Monthly
82 (1975), 813--817.
\end{thebibliography}
\end{document}